\documentclass{amsart}
\usepackage{graphicx} 
\usepackage[utf8]{inputenc}
\usepackage{amsfonts}
\usepackage{hyperref}
\DeclareMathOperator{\card}{card}

\hypersetup{pdftitle={A necessary condition for an F-space to be separable},
pdfsubject={Mathematics, },
pdfauthor={Hector},
pdfkeywords={}}

\usepackage{amsmath}
\usepackage{xcolor}
\usepackage{amsthm}
\usepackage{pdflscape}
\usepackage{pgfplots}
\usepackage{mathrsfs}
\usepackage{amssymb}
\newtheorem{thm}{Theorem}[section]
\newtheorem{cor}[thm]{Corollary}
\newtheorem{prop}[thm]{Proposition}

\newtheorem{quest}[thm]{Question}

\theoremstyle{definition}

\newtheorem{exmp}[thm]{Example}

\theoremstyle{remark}
\newtheorem{rem}[thm]{Remark}

\makeatletter
\let\c@equation\c@thm
\makeatother
\numberwithin{equation}{section}

\date{10 August 2026}

\title{A necessary condition for an F-space to be separable}

\begin{document}

\author{H\'ector N Salas }

\title{A necessary condition for  $F$-spaces to be separable}

\address{Department of Mathematical Sciences\\University of Puerto Rico, Mayag\"{u}ez, PR 00681-9018} \email{hector.salas@upr.edu}
\thanks{}
\date{}

\begin{abstract} 
A topological vector space  $E$ has the {\it countable intersection property for subspaces} if whenever
$\cap_{i\in I} M_i=\{0\}$ for a family of closed subspaces $M_i$ of $E,$ there is a countable set of indices $I_0\subset I$ such that $\cap_{i\in I_0} M_i=\{0\}.$ We prove that a separable $F$-space satisfies such a property, whereas 
a TVS with an uncountable Markushevich basis does not. We also exhibit a separable TVS that does not satisfy the property.
\end{abstract}

\maketitle

\section{Introduction}

Let TVS stand for a  topological vector space over the complex numbers $\mathbb C.$ The class of TVSs features a wide range of structural properties.
The spaces we consider are assumed to be Hausdorff and complete. All subspaces are assumed to be closed. 
An $F$-space is a complete and metrizable TVS whose
metric is translation-invariant. If, in addition, it has a local convex basis, then it is called a Fr\'echet space.

Let $E$ be a TVS. 

{\it Assume that whenever
$\cap_{i\in I} M_i=\{0\}$ for a family of closed subspaces $M_i$ of $E$ there is a countable set of indices $I_0\subset I$ such that $\cap_{i\in I_0} M_i=\{0\}.$
Then $E$ is said to have the countable intersection property for subspaces.} 

Our aim is to draw attention to this property that we will call ($C_0$).

Corson, in \cite{cor} defined the following property for a Banach space $E:$  

{\it Every collection of closed convex sets of $E$ with empty intersections has a countable subcollection  with an empty intersection. }

 Later, Pol \cite{pol} called it property $(C)$.
In the work of both authors and subsequent papers by other authors on $(C)$, weak and weak$^*$ topologies play an important role. This is also true for one characterization of separability for locally convex TVS, given by Ruf \cite{ruf} Theorem 2.5. 

However, for $(C_0)$ even the existence of the dual is not necessary. As a consequence of Theorem \ref{thm}, $H^p(\mathbb D)$ and $L^p([0,1],dx)$ with $0<p<1$ satisfy $(C_0).$ In the first case the dual does not contain enough functionals and the Hahn-Banach theorem fails, see Duren book \cite{dur}, section 7.5. In the second case, Day \cite{day} showed that there are no nontrivial functionals.

Also, in view of Theorem \ref{thm}, it is natural to ask if $C_0$ is satisfied whenever $E$
is a separable TVS.  A possible obstruction for this is a result of Doma\'nski \cite{dom} which shows a separable TVS with a non separable subspace. In Example 1 of \cite{klm}
K\c{a}kol, Leiderman, and  Morris
gave another example of that situation using the Michael line $\mathbb M$ which is the real line with the topology generated by its usual topology and all the singles $\{x\}$ with $x\in \mathbb P=\mathbb R\setminus \mathbb Q$. They showed that $C_p(\mathbb M),$ 
the continuous functions with the pointwise convergence, is  a separable locally convex space with a non separable subspace $C_p(\mathbb P).$

The paper is organized as follows. In Section 2, we present a separable space and several non-separable spaces failing property $(C_0).$ As a key tool for some of these constructions, we utilize the concept of a Markushevich basis. In Section 3, we prove the main theorem of the paper and provide several of its consequences. Finally, in Section 4, we pose a few open natural questions and show a possible extension of $(C_0)$ to $(C_1),~(C_2)$, and so forth.

Furthermore, these TVSs can be considered over the field of real numbers, with a few notable exceptions such as $H^p(\mathbb D)$.

\section{Examples and basic properties}

The result of \cite{klm} that is mentioned above allows us to show that there is a separably locally convex TVS that fails the countable intersection property for subspaces.

\begin{prop}  $C_p(\mathbb M)$ does not satisfy $(C_0).$
\end{prop}
\begin{proof}
For each $x\in \mathbb P$ let \[M_x=\{f\in C_p(\mathbb M):f(q)=0 ~\forall q \in \mathbb Q \text{ and } f(x)=0\}.\] Then $\cap_{x\in \mathbb P} M_x=\{0\} $
and any countable intersection is not $\{0\}.$\end{proof}

For a TVS $E$ the direct sum $E=G\bigoplus F$ 
means the usual, namely, $G\cap F=\{0\}$ and the projections $P_G(E)=G$ and $P_F(E)=F$ are continuous. If $E$ is a Hilbert space, it means orthogonal direct sum.

\begin{prop} 
(i) If $E$ satisfies ($C_0$), so does any closed subspace $F.$ 

(ii) If, in addition, $F$ is complemented in $E$, then the quotient $E/F$ also satisfies ($C_0$). \label{direct sum}
\end{prop}

\begin{proof}
(i) Let $F$ be a subspace of $E.$ Then $M$ be a subspace of $F$ implies that $M$ is also a subspace of $E.$

(ii) Let $E= G\bigoplus F.$ Then the relation between subspaces $E$ and $E/F$ is $\overline{M}=M\bigoplus F.$ Thus $\cap_{i\in I}\overline{M_i} =\{0\} $ translates to 
$\cap_{i\in I}{M_i} =\{0\}. $
\end{proof}
Let $T$ be a linear and bicontinuous one-to-one mapping from $E$ onto $G.$ Then $E$ satisfies ($C_0$) if and only if  $G$ does.

We provide below examples of TVS that do not satisfy ($C_0$).

Let $E$ be a TVS and $E^*$ be its dual. The pairs $x_i\in E$ and $x_i^*\in E^*$ form a Markushevich basis (M-basis) if

(a) Biorthogonality: $x_i^*(x_j)=\delta_{i,j}$

(b) Completeness: $E=\overline{\text{span}} \{x_i:i\in I\}.$

(c) Totality: $\cap_{i \in I} Ker(x_i^*)=\{0\}.$

\begin{rem} Examples of spaces with an M-basis where the index set  $I$ is uncountable.

(1) $\ell^p(I)=\{f:I\longrightarrow \mathbb C:\sum_{i\in I} |f(i)|^p<\infty \}.$ If $0<p<1$ the space is not locally convex, while if $1<p$ it is a Banach space.

(2) $c_0(I)=\{f:I\longrightarrow  \mathbb C: \forall \epsilon >0, \{i:|f(i)|>\epsilon \text{ is finite}\}\}$
\end{rem}

\begin{prop} \label{marku}If $E$ has an M-basis and $\card (I)> \aleph_0,$ then $E$ does not satisfy $(C_0).$  
\end{prop}

\begin{proof}
 Since $x_j\notin M_j=\overline{\text{span}}\{x_i:i\in I\setminus\{j\}\}$ and $x_j\in M_i$ if $i \neq j,$ we have $\cap_{i\in I} M_i=\{0\}$ but $x_j\in \cap_{i\in I_0} M_i$ if $I_0$ is countable and $j\notin I_0.$   
\end{proof}

In 1968, Amir and Lindenstraus \cite{amli} proved that every weakly compact generated Banach space admits a Projectional Resolution of the Identity.
They used this to show that all WCG spaces (and thus all reflexive spaces) admit a Markushevich basis. These results are also in the book by Fabian et al \cite{fhhmpz},  chapter 11.

\begin{cor} Every non separable reflexive Banach space does not satisfy $(C_0).$
\end{cor}

\begin{proof}
    The M-basis provided by \cite{amli} is uncountable. 
\end{proof}

Let $c_0(\mathbb N), ~\ell^1(\mathbb N),$ and $ \ell^{\infty}(\mathbb N)$ be denoted by $c_0, ~\ell^1,$ and $\ell^{\infty}$
respectively.

The non-reflexive spaces $c_0$ and $\ell^1$ satisfy $(C_0)$ since they are separable and  covered by Theorem   
 \ref{thm}, and they also possess a Schauder bases (and consequently M-bases), but we cannot use Proposition \ref{marku}.

The following spaces are non-reflexive. In particular, $\ell^{\infty},$ does not have an M-basis. This is a consequence of a result by Johnson \cite{joh}.

\begin{prop} (i) $L^{\infty}([0,1],dx),$ 
(ii) ~$\ell^{\infty},$ 
and (iii) ~ $\ell^{\infty} /c_0,$
do not satisfy ($C_0$). \label{L infty}
\end{prop}

\begin{proof} To see (i), consider for each $x\in [0,1],$
the ideal $M_x=\{f:\lim_{s\to x} f(s)=0\}.$ Thus \[\bigcap_{\{x:x\in[0,1]\}}M_x=\{0\}, \text{ but  }
\bigcap_{\{x_n:n\in  \mathbb N\}}M_{x_n}\neq \{0\}.\]
 Indeed, the characteristic function of $\chi_{[0,1]\setminus V}$ where $V$ is an open set containing $\{x_n:n\in  \mathbb N\}$ is in 
$\cap_{n \in \mathbb N}M_{x_n}, $ and such a function is not zero if the Lebesgue measure of $[0,1]\setminus V$ is positive.

(ii) The space  $L^{\infty}([0,1],dx)$ can be embedded isometrically into $\ell^{\infty}$ and vise versa. An application of a classical theorem by Pelczy\'nski says that $\ell^{\infty}$ and $L^{\infty}([0,1],dx)$ are  isomorphic.

(iii)   First, we give an especial embedding isometry
 \[T:L^{\infty}([0,1],dx)\longrightarrow  \ell^{\infty}. \]
 We defined  
 $T(f)(n)=\frac{1}{b_n-a_n}\int_{a_n}^{b_n}f(x)~dx$ where $f\in L^{\infty}([0,1],dx)$ and $([a_n,b_n])_{n=1}^{\infty}$ is an enumeration of all closed intervals whose extremes are rational numbers and
  $0\leq a_n<b_n\leq 1.$

  The Lebesgue differentiation theorem says that almost everywhere \[f(t)=\lim_{k \to \infty} \frac{1}{b_{n_k}-a_{n_k}}\int_{a_{n_k}}^{b_{n_k}}f(x)~dx\] whenever $t\in [a_{n_k},b_{n_k}]$ and $\lim_{k\to \infty}(b_{n_k}-a_{n_k})=0,$ but this means that
  $T(f)(n_k)\to f(t).$ In particular, if $t$ is one those points and $|f(t)|>c||f||_0$ we have that that there are infinitely many $k$ for which $|T(f)(n_k|>c||f||_0.$ (For any $0<c<1$.)
  
  Consequently $\Phi  \circ T $ is also an isometry, where $\Phi$ is the canonical projection from $\ell^{\infty}$ onto $\ell^{\infty}/c_0.$ 
  Then the  subspaces that are the image of the subspaces  in  part (i), 
   we keep the notation,  show that $\ell^{\infty}/c_0$ does not satisfy $C_0$ since
  \[\Phi  \circ T (\chi_{[0,1]\setminus V})\in   \bigcap_{\{x_n:n\in  \mathbb N\}}\Phi  \circ T(M_{x_n})~\text{  and  }~||\Phi  \circ T (\chi_{[0,1]\setminus V})||=1.\]
\end{proof}

\begin{cor} If a TVS contains a copy of 
$\ell^{\infty},$ then 
it does not satisfy $(C_0).$   
\end{cor}
\begin{proof}
    Apply Proposition \ref{direct sum} (i).
\end{proof}

\begin{rem}
  \label{haydon} If one wants to find a non-separable Banach space failing property $(C_0)$ (if such a space exists), a possible candidate is the space constructed by Haydon \cite{hay}. It would be interesting to determine if the exotic properties of this non-reflexive space, which does not contain a copy of  $\ell^{\infty}$, are sufficient to produce a failure of property $(C_0).$
\end{rem}

\section{The Theorem and Its Consequences}

\begin{thm} 
Let $E$ be a separable F-space. Then $E$ satisfies $(C_0),$ the countable intersection property for subspaces. \label{thm}
\end{thm}

\begin{proof} 
We proceed by contradiction. Suppose that there exists an uncountable index set $I$ and a family of closed subspaces $\{M_i\}_{i \in I}$ such that 
\begin{equation}
\bigcap_{i \in I} M_i = \{0\}, \label{contradiction1}
\end{equation} 
and yet, for every countable subset $I_0 \subset I$, 
\begin{equation} 
\bigcap_{i \in I_0} M_i \neq \{0\}. \label{contradiction2} 
\end{equation} 
Let $\omega_1$ denote the first uncountable ordinal. By transfinite induction, we construct a strictly decreasing sequence of closed subspaces $\{N_\alpha\}_{\alpha < \omega_1}$ where each $N_\alpha$ is a countable intersection of elements from $\{M_i\}_{i \in I}$.

For the base case, choose an arbitrary $i_0 \in I$ and define $N_0 = M_{i_0}$. Now, let $\beta < \omega_1$ and assume that $N_\alpha$ has been defined for all $\alpha < \beta$ such that each $N_\alpha$ is a countable intersection of elements from $\{M_i\}_{i \in I}$. 

We split the induction step into two cases based on the nature of $\beta$:
\begin{itemize}
    \item {Case 1 ($\beta$ is a successor ordinal, $\beta = \gamma + 1$):} By hypothesis, $N_\gamma \neq \{0\}$. Due to \eqref{contradiction1}, $N_\gamma$ cannot be contained in every $M_i$ for $i \in I$. Thus, there exists $i_\beta \in I$ such that $N_\gamma \not\subset M_{i_\beta}$. We define $N_\beta = N_\gamma \cap M_{i_\beta}$, yielding the strict inclusion $N_\beta \subsetneqq N_\gamma$.
    \item {Case 2 ($\beta$ is a limit ordinal):} Define $N' = \bigcap_{\alpha < \beta} N_\alpha$. Since $\beta$ is a countable ordinal, $N'$ is a countable intersection of the original subspaces, so $N' \neq \{0\}$ by \eqref{contradiction2}. By the same argument as in Case 1, there exists $i_\beta \in I$ such that $N' \not\subset M_{i_\beta}$. Defining $N_\beta = N' \cap M_{i_\beta}$ ensures that $N_\beta \subsetneqq N'$ and consequently $N_\beta \subsetneqq N_\alpha$ for all $\alpha < \beta$.
\end{itemize}
By construction, $\{N_\alpha\}_{\alpha < \omega_1}$ is a strictly decreasing transfinite sequence of closed subspaces of $E$.

For each $\alpha < \omega_1$, choose an element $x_{\alpha} \in N_{\alpha} \setminus N_{\alpha+1}$. Let $d$ be an invariant metric compatible with the topology of the F-space $E$. Because each $N_{\alpha + 1}$ is closed, the distance from $x_\alpha$ to $N_{\alpha+1}$ is strictly positive:
\[f(\alpha) = \inf\{d(x_{\alpha}-x) : x \in N_{\alpha + 1}\} > 0.\] 
Since $\omega_1$ is uncountable, the  regularity of $\omega_1$ implies there exists a constant $c > 0$ such that the set $J = \{\alpha < \omega_1 : f(\alpha) \geq c\}$ is uncountable. 

Let $\alpha, \beta \in J$ be distinct indices. Assuming without loss of generality that $\alpha < \beta$, we have $x_{\beta} \in N_{\beta} \subset N_{\alpha + 1}$. By the definition of infimum, it follows that $d(x_{\alpha}-x_{\beta}) \geq f(\alpha) \geq c$. 

Since $E$ is separable, it contains a countable dense subset $Y$. For each $\alpha \in J$, we may choose an element $y_{\alpha} \in Y$ satisfying $d(x_{\alpha}-y_{\alpha}) < \frac{c}{3}$. By the triangle inequality, for any distinct $\alpha, \beta \in J$, we obtain: 
\[d(y_{\alpha}-y_{\beta}) \geq d(x_{\alpha}-x_{\beta}) - d(x_{\alpha}-y_{\alpha}) - d(x_{\beta}-y_{\beta}) \geq  \frac{c}{3}.\] 
This implies that the mapping $\alpha \mapsto y_{\alpha}$ is an injective function from the uncountable set $J$ into the countable set $Y$, which is a contradiction. This completes the proof. 
\end{proof}
As an immediate consequence of this theorem, we obtain:
\begin{cor}
Separable Fr\'echet spaces, and consequently separable Banach spaces, satisfy ($C_0$).
\label{separable frechet}
\end{cor}

A TVS $E$ is said to be an LF space if it is  a countably inductive  limit of Fr\'echet spaces, Tr\`eves \cite{tre} chapter 13.
The following corollary  broadens the scope of our main result, showing that the conclusion holds even when $E$ is not metrizable. This occurs whenever $E_n\subsetneq E_{n+1}~\forall n.$ 
\begin{cor} Let $E$ be an LF-space, countable inductive limit of separable Fr\'echet spaces. Then $E$ is separable and satisfies ($C_0$).
\end{cor}

\begin{rem}
 For Hilbert spaces, separable is equivalent to satisfying $(C_0).$ This is immediate since $\oplus$ can be interpreted as an orthogonal direct sum.  
\end{rem}

Let $H$ be a separable Hilbert space and $\mathcal B(H)$ be the bounded linear operators on $H.$

\begin{cor} Let $A$ be a $C^*$
algebra contained in $\mathcal B(H).$ If $\cap_{T\in A} Ker T=\{0\},$ then there exists an injective $S\in A.$ Moreover, $T$
 can be chosen to be a positive linear operator.
\end{cor}

\begin{proof} Since $Ker(T)=Ker \big(\sqrt{T^*T}\big)$ and $\sqrt{T^*T} \in A,$ we have 
\[\bigcap_{\{T \in A, 0\leq T\}} Ker (T)=\{0\}.\] However, $H$ satisfies ($C_0$), and consequently there exists
$T_n \in A$ with $0\leq T_n ~\forall n$ such that
$\cap_{\{n \in \mathbb N\}} Ker( T_n)=\{0\}.$ Set $S=\sum_{n=1}\frac{2^{-n}}{||T_n||}T_n.$ Then
$\langle Sx,x\rangle=0$ implies that $\langle T_nx,x\rangle=0$ and so $T_n(x)=0 ~\forall n$ and therefore $x=0.$
\end{proof}

\begin{rem}
(i) For the proof to work, it suffices that $A\subset\mathcal B(H)$ be a Banach space, and $\cap_{\{T \in A, 0< T\}} Ker (T)=\{0\}.$

(ii) If $A$ is not a $C^*$ algebra, the existence of an injective operator cannot be ensured. For example, let $T$ be the unilateral backward shift, $T((x_1,x_2,\cdots))=(x_2,x_3<\cdots)$ for $x=(x_1,x_2, \cdots) \in \ell^2(\mathbb N).$ The algebra generated by $T,T^2,T^3, \cdots$
does not contain an injective operator although $\cap_n Ker(T^n)=\{0\}.$

(iii) If $H$ is not separable and $A=\mathcal K(H)$ is the $C^*$ algebra generated by the compact operators, then $\cap_{\{T \in \mathcal K(H)\}} Ker (T)=\{0\}.$

However, for $T$ compact, $Ker(T)$ has separable codimension and 
$\cap_{n\in\mathbb N} Ker(T_n)$ also has separable codimension and therefore is not $\{0\}.$

\end{rem}

The following proposition is a partial converse for Proposition \ref{marku}. 

\begin{prop} \label{allseparable}
Assume that $E$ is a separable TVS in which every pair of closed subspaces $X\subsetneq Y$   is such that $X$ has an M-basis that can be extended to an M-basis  of $Y.$ Then $E$ satisfies $(C_0).$
\end{prop}

\begin{proof}
By choosing $X=\{0\}$ and $Y=E$ we have an M-basis which is countable.

The proof now follows the outline of the proof of Theorem \ref{thm} and the contradiction will be that it is possible to build an uncountable M-basis for $E.$
\end{proof}

\section{ Questions and a possible extension} 

We present an example to motivate and frame the first question; its part (iii) relates to Remark \ref{haydon}. The other questions test the limits of our results.

The concluding remark points to possible
$(C_{\alpha})$ extensions for $0<\alpha.$
If $E$ is a TVS, the requirement is that its density character is at least $\aleph_{\alpha}.$ 

\begin{quest} 

(i) Is it true that any non separable TVS does not satisfy ($C_0$)? 

(ii) What if, in addition, it is also locally convex?

(iii) What if it is a Banach space?
\end{quest}

The following direct approach does not seem to work for (ii): Let $E$ be a locally convex TVS, and for each non-zero $x \in E$, consider $\varphi_x \in E^*$ with $\varphi_x(x) = 1$, so that 
\[E = Ker(\varphi_x) \bigoplus \text{span}(x).\] Then $\cap_{x \in E} Ker(\varphi_x) = \{0\}$, and the goal is to show that for any sequence 
\[\{x_n\}_{n \in \mathbb{N}} \subset E, \text{ we have } \bigcap_{n \in \mathbb{N}} Ker(\varphi_{x_n}) \neq \{0\}.\]
The following example illustrates the issue. Although Proposition \ref{L infty} (ii) shows that $\ell^{\infty}$ fails $(C_0),$ the approach outlined above cannot establish this fact.

\begin{exmp} Let $E=\ell^{\infty}$ and $\varphi_n(x)=x_n$
where $x=(x_k)_{k=1}^{\infty}.$
    \[Ker(\varphi_n)=M_{\varphi_n}=\{x\in \ell^{\infty} :x_n =0\}\] \[\text{ but }  \bigcap_{n\in \mathbb N} M_{\varphi_n}=\{0\}.\]
\end{exmp}

\begin{quest} What are the separable TVSs that do not satisfy $(C_0)?$
\end{quest}

\begin{quest}
Assume that $E$ is a separable TVS and all its subspaces are separable. Does $E$ satisfy $(C_0)?$
\end{quest}

$E$ could certainly have an uncountable chain of subspaces. For example
$M_a=\chi_{(a,1]}L^1([0,1],dx)$ where $0<a<1.$ 
\begin{quest}
Are there any examples of a TVS obtained via Prop \ref{allseparable} but not obtainable as a corollary of Theorem \ref{thm}?
\end{quest}

Another natural direction is to extend property $(C_0)$ to arbitrary cardinals. 
Thus, the next step is to consider a TVS $E$ whose density character is at least $\aleph_1.$ In this context, property $(C_1)$ 
means that if $\cap_{i \in I} M_i=\{0\}$ implies that there exist $I_1\subset I$ with $\card(I_1)\leq \aleph_1$ and $\cap_{i \in I_1} M_i=\{0\}$ whenever $M_i$ are closed subspaces of $E.$ 

Theorem \ref{thm} naturally extends to this new setting, with essentially the same proof. Let $\omega_2$ the first ordinal such that $\card(\omega_2)=\aleph_2.$ We have to consider $\omega_2$ instead of $\omega_1$ and use the fact that for any collection 
\[\{c_i:i\in I \text{ and } c_i>0\}\]    with  $\card(I)=\aleph_2,$ there is a subset $J\subset I$ with $\card(J)=\aleph_2$ and 
\[0<\text{inf}\{c_i:i\in J\}.\]

Proposition \ref{marku} can also be adapted when it is required that the cardinality of the M-basis be at least $\aleph_2.$


\begin{thebibliography}{9}

\bibitem{amli} Amir, D. ; Lindenstrauss, J. \emph{The structure of weakly compact subsets in Banach spaces.} Ann. of Math. (2) {\bf 88} , 35–46(1968) .


\bibitem {cor} Corson, H. H. 
\emph {The weak topology of a Banach space.} 
Trans. Am. Math. Soc. {\bf 101}, 1-15 (1961).

\bibitem{day}  Day, M.M. \emph{ The spaces 
 $L^p$ with  $0<p<1.$}
Bull. Amer. Math. Soc. {\bf 46}, 816-823 (1940).

\bibitem {dom} Doma\'nski, P.
\emph{Nonseparable closed subspaces in separable products of topological vector spaces, and q-minimality.} 
Arch. Math. {\bf 41}, 270–275 (1983).

\bibitem{dur} Duren, P.
\emph{Theory of $H^p$ spaces.} Pure and applied mathematics {\bf 38} Academic Press. New York 1970. Dover reprint (2000).

\bibitem{fhhmpz} Fabian, M.; Habala, P.; Hájek, P.; Montesinos Santalucía, V.; Pelant, J.; Zizler, V. 
\emph{Functional Analysis and Infinite- Dimensional Geometry. }

CMS Books in Mathematics/Ouvrages de Mathématiques de la SMC. {\bf 8}. New York, NY: Springer. ix, 451 p. (2001).

\bibitem{hay} Haydon, R.
\emph{A non-reflexive Grothendieck space that does not contain $\ell^{\infty}.$}
Israel J. of Math, Vol. {\bf 40,} No. 1, 1981

\bibitem{joh}
Johnson, W. B.
\emph{No infinite dimensional P space admits a Markuschevich basis.} 
Proc. Am. Math. Soc. {\bf 26,} 467-468 (1970).

\bibitem{klm} K\c{a}kol, J.;  Leiderman, A.G.;  Morris, S.A. \emph{Nonseparable closed vector subspaces of separable topological vector spaces.}
{\bf 182}, 39–47 (2017)




\bibitem {pol} Pol, R.
\emph{ On a question of H. H. Corson and some related problems.} 
Fundam. Math. {\bf 109}, 143–154

\bibitem{ruf} Ruf, T. \emph{Separability and submetrizability in locally convex spaces.} 
  Annals of Functional Analysis
{\bf 17}, article number 48 (2026)

\bibitem {tre} Tr\`eves, F. \emph{Topological vector spaces, distributions and kernels.}
Pure and Applied Mathematics (Academic Press) 25. New York-London: Academic Press. XVI, 565 p. (1967).

 \end{thebibliography}
\end{document}